\documentclass[a4paper]{amsart}
\usepackage[T1]{fontenc}
\usepackage[british]{babel}
\usepackage{amssymb}
\usepackage{mathtools}
\DeclareSymbolFont{bbold}{U}{bbold}{m}{n}
\DeclareSymbolFontAlphabet{\mathbbm}{bbold}
\newcommand{\rest}{\mathbin{\upharpoonright}}
\newcommand{\B}{\mathbb{B}}
\newcommand{\C}{\mathbb{C}}
\newcommand{\bbot}{\mathbbm{0}}
\newcommand{\btop}{\mathbbm{1}}
\newcommand{\modulo}[1]{({\mathgroup\symoperators mod}\mkern6mu#1)}
\DeclareMathOperator{\dom}{dom}
\DeclareMathOperator{\Part}{Part}
\DeclareMathOperator{\End}{End}
\DeclareMathOperator{\cof}{cof}
\DeclarePairedDelimiter{\abs}{\lvert}{\rvert}
\DeclarePairedDelimiterX{\Set}[2]{\{}{\}}{ #1 \nonscript\:\delimsize\vert\allowbreak\nonscript\:\mathopen{} #2 }
\DeclarePairedDelimiterX{\Seq}[2]{\langle}{\rangle}{ #1 \nonscript\:\delimsize\vert\allowbreak\nonscript\:\mathopen{} #2 }
\theoremstyle{plain}
\newtheorem{theorem}{Theorem}[section]
\newtheorem{proposition}[theorem]{Proposition}
\newtheorem{lemma}[theorem]{Lemma}
\theoremstyle{definition}
\newtheorem{definition}[theorem]{Definition}
\newtheorem{question}[theorem]{Question}
\theoremstyle{remark}
\newtheorem{remark}[theorem]{Remark}
\begin{document}
\title{Orderings of ultrafilters on Boolean algebras}
\author{J\"org Brendle}
\address{Graduate School of System Informatics\\Kobe University\\1-1 Rokkodai-cho\\Nada-ku\\Kobe 657-8501\\Japan}
\author{Francesco Parente}
\address{Institut f\"ur Diskrete Mathematik und Geometrie\\Technische Universit\"at Wien\\Wiedner Hauptstra{\ss}e 8-10/104\\1040 Vienna\\Austria}
\dedicatory{Dedicated to the memory of Kenneth Kunen (1943--2020)}
\thanks{The first author is partially supported by Grant-in-Aid for Scientific Research (C) 18K03398, Japan Society for the Promotion of Science. The second author was funded by an International Research Fellowship of the Japan Society for the Promotion of Science, as well as Austrian Science Fund (FWF) project P33420.}
\begin{abstract}
We study two generalizations of the Rudin-Keisler ordering to ultrafilters on complete Boolean algebras. To highlight the difference between them, we develop new techniques to construct incomparable ultrafilters in this setting. Furthermore, we discuss the relation with Tukey reducibility and prove that, assuming the Continuum Hypothesis, there exist ultrafilters on the Cohen algebra which are RK-equivalent in the generalized sense but Tukey-incomparable, in stark contrast with the classical setting.
\end{abstract}
\maketitle

\section{Introduction}
The Rudin-Keisler ordering is a basic tool in the study of ultrafilters over $\omega$. In 1972, Kunen \cite{kunen:ultrafilters} proved that the ordering is not linear, that is, there exist ultrafilters which are RK-incomparable. The method of independent families, used by Kunen in the proof, has since then become widespread in set theory and general topology. Recently, the interplay between the Rudin-Keisler ordering, Tukey reducibility, and combinatorial properties of ultrafilters has re-emerged as a fertile topic of research, surveyed in \cite{dobrinen:tukey}.

In 1999, Murakami \cite{murakami:rk} extended the definition of the Rudin-Keisler ordering to ultrafilters on complete Boolean algebras: his notion, which we call the M-ordering, is formulated in terms of complete homomorphisms. Two years later Jipsen, Pinus, and Rose \cite{jpr:rk} proposed independently a second generalization, motivated by model-theoretic considerations, which we call the JPR-ordering. Both extensions coincide with the classical ordering if the algebra is atomic but, as far as we know, no work so far has examined their relation in the general setting. In Section \ref{section:due}, to which we refer the reader for precise definitions, we introduce the aforementioned concepts and observe that the M-ordering is stronger than both the JPR-ordering and Tukey reducibility.

In Section \ref{section:tre}, we extend Kunen's method of independent families to a wide class of complete Boolean algebras: as a result, we establish in ZFC the existence of ultrafilters which are Tukey-equivalent but JPR-incomparable. Section \ref{section:quattro}, on the other hand, is dedicated to showing that the two extensions of the RK-ordering need not coincide: in fact, we build two ultrafilters which are JPR-equivalent but M-incomparable. Our argument relies on the Continuum Hypothesis and is applicable to several Boolean algebras, such as the Cohen and random forcing.

In the final section, we restrict our attention to the Cohen algebra, on which we construct ultrafilters which are JPR-equivalent but Tukey-incomparable. This configuration is strikingly different from atomic algebras, where the RK-ordering is known to be stronger than Tukey reducibility. The main ingredient of our construction is the analysis of Tukey reductions on coherent $P$-ultrafilters, for which we establish a canonical representation theorem parallel to Dobrinen and Todor\v{c}evi\'c \cite{dt:tukey}.

\section{The orderings}\label{section:due}

In this section, we introduce the orderings under investigation. To do so, we assume some familiarity with the language of Boolean algebras. Suppose $\B$ is a Boolean algebra; given $B\subseteq\B$ and $x\in B$, we define
\[
B\rest x=\Set{b\in B}{b\le x}.
\]
If $A$ and $B$ are two antichains in $\B$, we say that $B$ is a \emph{refinement} of $A$ if for every $b\in B$ there exists $a\in A$ such that $b\le a$. Finally, let $\Part(\B)$ be the partially ordered set of all maximal antichains in $\B$ equipped with the refinement relation.

The starting point is the ordering introduced in the sixties by Rudin \cite{rudin:ultrafilters} and Keisler \cite{keisler:notes} independently: given two ultrafilters $U$ and $V$ over a set $A$, we say that \emph{$U\le_\mathrm{RK} V$} if there exists a function $f\colon A\to A$ such that for all $X\subseteq A$
\[
X\in U\iff f^{-1}[X]\in V.
\]
All the orderings we consider are in fact reflexive transitive relations. Accordingly, if both $U\le V$ and $V\le U$ then we shall say that $U$ and $V$ are \emph{equivalent}, in symbols $U\equiv V$.

The RK-ordering has a model-theoretic characterization, due to Blass \cite[Proposition 11.7]{blass:thesis}, and a Boolean-algebraic reformulation, both of which are summarized in the next proposition.

\begin{proposition}\label{proposition:blass} Let $U$ and $V$ be ultrafilters over a set $A$; then the following conditions are equivalent:
\begin{enumerate}
\item $U\le_\mathrm{RK} V$;
\item\label{blassdue} there exist $Y\in V$ and a complete homomorphism $h\colon\mathcal{P}(A)\to\mathcal{P}(Y)$ such that $U=h^{-1}[V]$;
\item\label{blasstre} for every structure $\mathfrak{M}$, there exists an elementary embedding $\mathfrak{M}^A\!/U\to\mathfrak{M}^A\!/V$.
\end{enumerate}
\end{proposition}

In this paper, we shall focus on two generalizations to complete Boolean algebras of \eqref{blassdue} and \eqref{blasstre} respectively. As both have appeared in the literature under the name ``Rudin-Keisler ordering'', to avoid confusion we have no choice but to temporarily rename them. Let us move on to present the first Boolean-algebraic formulation.

\begin{definition}[Murakami \cite{murakami:rk}]\label{definition:m} Let $U$ and $V$ be ultrafilters on complete Boolean algebras $\B$ and $\C$, respectively. We say that \emph{$U\le_\mathrm{M} V$} if there exist $v\in V$ and a complete homomorphism $h\colon\B\to\C\rest v$ such that $U=h^{-1}[V]$.
\end{definition}

Clearly reminiscent of condition \eqref{blassdue} of Proposition \ref{proposition:blass}, the M-ordering was originally introduced in the framework of nonstandard universes. Two years later, however, a second generalization was proposed independently.

\begin{definition}[{Jipsen, Pinus, and Rose \cite[Definition 1.1]{jpr:rk}}]\label{definition:jpr} Let $U$ and $V$ be ultrafilters on complete Boolean algebras $\B$ and $\C$, respectively. We say that \emph{$U\le_\mathrm{JPR} V$} if there exist a function $g\colon\Part(\B)\to\Part(\C)$ and, for each $A\in\Part(\B)$, a function $f_A\colon g(A)\to A$ such that:
\begin{enumerate}
\item\label{eq:jpruno} for every $A\in\Part(\B)$ and every $X\subseteq A$,
\[
\bigvee X\in U\iff\bigvee f_A^{-1}[X]\in V;
\]
\item\label{eq:jprdue} if $A,B\in\Part(\B)$ and $B$ is a refinement of $A$, then
\[
\bigvee\Set{a\wedge b}{a\in g(A)\text{, }b\in g(B)\text{, and }f_{B}(b)\le f_A(a)}\in V.
\]
\end{enumerate}
\end{definition}

Although it may not be evident at first glance, the characterization in \cite[Theorem 2.4]{jpr:rk} shows that the JPR-ordering is indeed a generalization of point \eqref{blasstre} of Proposition \ref{proposition:blass}, provided that ultrapowers are replaced with Boolean ultrapowers.

\begin{remark} From now on, we shall only be interested in comparing ultrafilters $U$ and $V$ over the same set or, more generally, on the same complete Boolean algebra.
\end{remark}

Two sufficient conditions for the JPR-ordering are sometimes easier to check: we collect them in the following lemma.

\begin{lemma}[Jipsen, Pinus, and Rose \cite{jpr:rk}]\label{lemma:dense} Let $U$ and $V$ be ultrafilters on a complete Boolean algebra $\B$.
\begin{enumerate}
\item\label{lemma:jpruno} If the conditions of Definition \ref{definition:jpr} are satisfied for all $A$ in a dense subset of $\Part(\B)$, then $U\le_\mathrm{JPR} V$.\item\label{lemma:jprdue} If there exist $u\in U$ and $v\in V$ such that $U\rest u\le_\mathrm{JPR} V\rest v$, then $U\le_\mathrm{JPR} V$.
\end{enumerate}
\end{lemma}

To sum up, by Proposition \ref{proposition:blass} the M-ordering and the JPR-ordering coincide with the classical RK-ordering on atomic algebras, but to what extend can they differ in general? From what is stated so far, it already follows that one ordering is stronger than the other.

\begin{proposition} Let $U$ and $V$ be ultrafilters on a complete Boolean algebra $\B$. If $U\le_\mathrm{M} V$ then $U\le_\mathrm{JPR} V$.
\end{proposition}
\begin{proof} Suppose $U\le_\mathrm{M} V$, as witnessed by a complete homomorphism $\B\to\B\rest v$. Then \cite[Proposition 1.2]{jpr:rk} implies that $U\le_\mathrm{JPR} V\rest v$, hence $U\le_\mathrm{JPR} V$ by Lemma \ref{lemma:dense}\eqref{lemma:jprdue}.
\end{proof}

Let us now introduce the notion of \emph{Tukey reducibility} \cite{tukey:types}, which will play a role in Sections \ref{section:tre} and \ref{section:cinque}.

\begin{definition} Let $U$ and $V$ be ultrafilters on a Boolean algebra $\B$. We define $\langle U,\ge\rangle\le_{\mathrm{T}}\langle V,\ge\rangle$ if and only if there exist functions $f\colon U\to V$ and $g\colon V\to U$ such that for all $u\in U$ and $v\in V$
\[
v\le f(u)\implies g(v)\le u.
\]
\end{definition}

Recall that $C\subseteq U$ is \emph{cofinal} if for every $u\in U$ there exists $c\in C$ such that $c\le u$. As usual, we let $\cof(U)$ be the minimum cardinality of a cofinal subset of $U$.

\begin{remark}\label{remark:monotonic} If a pair of functions $\langle f,g\rangle$ witnesses Tukey reducibility as above, then it is easy to verify that, whenever $C\subseteq V$ is cofinal, its pointwise image $g[C]$ is cofinal in $U$. Furthermore, as observed by Isbell \cite{isbell:cofinal}, the function $g$ may be taken to be \emph{monotonic}, in the sense that if $v_1\le v_2$ then $g(v_1)\le g(v_2)$.
\end{remark}

For ultrafilters over $\omega$, if $U\le_\mathrm{RK} V$ then $\langle U,\supseteq\rangle\le_\mathrm{T}\langle V,\supseteq\rangle$, which is a well-known fact proved in Dobrinen and Todor\v{c}evi\'c \cite[Fact 1]{dt:tukey}. Indeed, the same is true in the general context for the M-ordering.

\begin{proposition}[{\cite[Proposition 2.7]{bp:comb}}]\label{proposition:27} Let $U$ and $V$ be ultrafilters on a complete Boolean algebra $\B$. If $U\le_\mathrm{M} V$ then $\langle U,\ge\rangle\le_{\mathrm{T}}\langle V,\ge\rangle$.
\end{proposition}

Thus, the M-ordering is stronger than both the JPR-ordering and Tukey reducibility. In the remainder of this paper, we shall see that in general no further relations hold between them.

\section{JPR-incomparable ultrafilters}\label{section:tre}

This section is dedicated to the construction of ultrafilters which are Tukey-equivalent but JPR-incomparable. Our argument extends Kunen's technique of independent families to a wider class of Boolean algebras.

First, to ensure Tukey-equivalence, we shall make use of the following criterion (see Dobrinen and Todor\v{c}evi\'c \cite[Fact 12]{dt:tukey}).

\begin{lemma}\label{lemma:max} For an ultrafilter $V$ on a complete Boolean algebra $\B$, the following conditions are equivalent:
\begin{itemize}
\item for every ultrafilter $U$ on $\B$ we have $\langle U,\ge\rangle\le_{\mathrm{T}}\langle V,\ge\rangle$;
\item there exists a subset $X\subseteq V$ with $\abs{X}=\abs{\B}$ such that every infinite $Y\subseteq X$ has $\bigwedge Y\notin V$.
\end{itemize}
\end{lemma}

Let us stipulate that an ultrafilter on $\B$ is \emph{Tukey-maximal} if it satisfies the equivalent conditions of Lemma \ref{lemma:max}. In particular, two Tukey-maximal ultrafilters are necessarily equivalent.

Second, to obtain JPR-incomparability, we proceed as in the original proof of Kunen \cite[Theorem 2.2]{kunen:ultrafilters}, keeping a large independent subset to guarantee the recursive construction does not halt.

\begin{definition} Let $F$ be a filter on a Boolean algebra $\B$. A subset $S\subseteq\B$ is independent $\modulo{F}$ if, whenever $s_1,\dots,s_m,t_1,\dots,t_n\in S$ are distinct,
\[
\neg s_1\vee\dots\vee\neg s_m\vee t_1\vee\dots\vee t_n\notin F.
\]
\end{definition}

Although more general existence results are available, such as Balcar and Franek \cite[Theorem A]{balfran:cba}, the following suffices for our purpose.

\begin{lemma}[Hausdorff \cite{hausdorff:indep}]\label{lemma:indep} For every infinite set $A$ there exists a family $S\subseteq\mathcal{P}(A)$ such that $\abs{S}=2^{\abs{A}}$ and $S$ is independent $\modulo{\{A\}}$.
\end{lemma}

\begin{remark}\label{remark:bound} To understand the hypothesis of the next theorem, recall that if $\B$ is a complete Boolean algebra and $A$ is an antichain in $\B$, then $2^{\abs{A}}\le\abs{\B}$.
\end{remark}

\begin{theorem} Let $\B$ be an infinite complete Boolean algebra. Suppose there exists an antichain $A$ in $\B$ such that $2^{\abs{A}}=\abs{\B}$. Then there exist two Tukey-maximal ultrafilters $U$ and $V$ on $\B$ such that $U\nleq_\mathrm{JPR} V$ and $V\nleq_\mathrm{JPR} U$.
\end{theorem}
\begin{proof} Without loss of generality, $A$ is a maximal antichain; let $\kappa$ be its cardinality, so that $\abs{\B}=2^\kappa$ by hypothesis. Let us define
\[
A^{[\B]}=\Set*{f\colon P\to A}{P\in\Part(\B)}
\]
and note that
\[
\abs*{A^{[\B]}}\le\abs{\Part(\B)}\cdot\sup\Set*{{\kappa}^{\abs{P}}}{P\in\Part(\B)}\le\sup\Set*{{(2^\kappa)}^{\abs{P}}}{P\in\Part(\B)}= 2^\kappa,
\]
where the equality on the right follows from Remark \ref{remark:bound}, which gives $2^{\abs{P}}\le 2^\kappa$ for every $P\in\Part(\B)$. Hence, we can enumerate
\[
A^{[\B]}=\Set{f_\alpha}{\alpha<2^\kappa}.
\]

We shall carry out a recursive construction of filters $\Seq{U_\alpha}{\alpha<2^\kappa}$ and $\Seq{V_\alpha}{\alpha<2^\kappa}$ on $\B$, together with subsets $\Seq{S_\alpha}{\alpha<2^\kappa}$ of $\B$, according to the following conditions:
\begin{enumerate}
\item\label{kunenuno} $\abs{S_0}=2^\kappa$, $U_0=V_0$, and there exists a subset $S'\subseteq U_0$ with $\abs{S'}=2^\kappa$ such that every infinite $I\subseteq S'$ has $\neg\bigwedge I\in U_0$;
\item if $\alpha<\beta<2^\kappa$ then $U_\alpha\subseteq U_\beta$, $V_\alpha\subseteq V_\beta$, and $S_\beta\subseteq S_\alpha$;
\item\label{kunentre} if $\delta<2^\kappa$ is a limit ordinal, then $U_\delta=\bigcup_{\alpha<\delta}U_\alpha$, $V_\delta=\bigcup_{\alpha<\delta}V_\alpha$, and $S_\delta=\bigcap_{\alpha<\delta}S_\alpha$;
\item\label{kunenquattro} for every $\alpha<2^\kappa$, the set $S_\alpha\setminus S_{\alpha+1}$ is finite;
\item\label{kunencinque} for every $\alpha<2^\kappa$, $S_\alpha$ is independent $\modulo{U_\alpha}$ and independent $\modulo{V_\alpha}$;
\item\label{kunensei} for $\alpha<2^\kappa$ there exists $X\subseteq A$ such that $\bigvee X\in U_{\alpha+1}$ and $\bigvee f_\alpha^{-1}[A\setminus X]\in V_{\alpha+1}$, and similarly there exists $Y\subseteq A$ such that $\bigvee Y\in V_{\alpha+1}$ and $\bigvee f_\alpha^{-1}[A\setminus Y]\in U_{\alpha+1}$.
\end{enumerate}

We begin by constructing $U_0$, $V_0$, and $S_0$. By Lemma \ref{lemma:indep}, let $S\subseteq\mathcal{P}(A)$ be such that $\abs{S}=2^\kappa$ and $S$ is independent $\modulo{\{A\}}$. Now partition
\[
\Set*{\bigvee X}{X\in S}=S_0\cup S'
\]
in such a way that $\abs{S_0}=\abs{S'}=2^\kappa$. Let $U_0$ be the filter on $\B$ generated by
\[
S'\cup\Set*{\neg\bigwedge I}{I\subseteq S'\text{ is infinite}}.
\]
We now prove that $S_0$ is independent $\modulo{U_0}$, thus showing at the same time that $U_0$ is indeed a proper filter. So let $s_1,\dots,s_m,t_1,\dots,t_n\in S_0$ be distinct and suppose towards a contradiction that
\[
\neg s_1\vee\dots\vee\neg s_m\vee t_1\vee\dots\vee t_n\in U_0.
\]
This means we can find elements $s'_1,\dots,s'_k\in S'$ and infinite subsets $I_1,\dots,I_l\subseteq S'$ such that
\[
s'_1\wedge\dots\wedge s'_k\wedge\neg\bigwedge I_1\wedge\dots\wedge\neg\bigwedge I_l\le\neg s_1\vee\dots\vee\neg s_m\vee t_1\vee\dots\vee t_n.
\]
For each $1\le j\le l$, let us choose $i_j\in I_j\setminus\{s'_1,\dots,s'_k\}$. In conclusion, we have
\[
s'_1\wedge\dots\wedge s'_k\wedge\neg i_1\wedge\dots\wedge\neg i_l\wedge s_1\wedge\dots\wedge s_m\wedge\neg t_1\wedge\dots\wedge\neg t_n=\bbot,
\]
which is easily seen to be in contradiction with the fact that $S$ is independent $\modulo{\{A\}}$. Concluding the base step of the construction, let $V_0=U_0$, satisfying \eqref{kunenuno} and the base case of \eqref{kunencinque}.

\begin{lemma}\label{lemma:kunen} Let $F$ and $G$ be filters on $\B$; suppose $T\subseteq S_0$ is infinite, independent $\modulo{F}$, and independent $\modulo{G}$. Then for every $f\in A^{[\B]}$ there exist $T'\subseteq T$ and filters $F'\supseteq F$, $G'\supseteq G$ such that:
\begin{itemize}
\item $T\setminus T'$ is finite;
\item $T'$ is independent $\modulo{F'}$ and independent $\modulo{G'}$;
\item there exists $Z\subseteq A$ such that $\bigvee Z\in F'$ and $\bigvee f^{-1}[A\setminus Z]\in G'$.
\end{itemize}
\end{lemma}
\begin{proof} Take any $t\in T$. Note that, as $T\subseteq S_0\subseteq\Set{\bigvee X}{X\in S}$, there exists $W\subseteq A$ such that $t=\bigvee W$. Now we simply follow Kunen \cite[Lemma 2.6]{kunen:ultrafilters}.

Suppose $T\setminus\{t\}$ is independent modulo the filter generated by $G\cup\bigl\{\bigvee f^{-1}[A\setminus W]\bigr\}$. Then we let
\begin{itemize}
\item $T'=T\setminus\{t\}$;
\item $F'$ be the filter generated by $F\cup\{t\}$ and $G'$ be the filter generated by $G\cup\bigl\{\bigvee f^{-1}[A\setminus W]\bigr\}$, which are again proper filters by independence;
\item $Z=W$.
\end{itemize}

On the other hand, suppose $T\setminus\{t\}$ is \emph{not} independent modulo the filter generated by $G\cup\bigl\{\bigvee f^{-1}[A\setminus W]\bigr\}$. This means there exist distinct $s_1,\dots,s_m,t_1,\dots,t_n\in T\setminus\{t\}$ and $g\in G$ such that
\[
g\wedge\bigvee f^{-1}[A\setminus W]\le\neg s_1\vee\dots\vee\neg s_m\vee t_1\vee\dots\vee t_n.
\]
Then we let
\begin{itemize}
\item $T'=T\setminus\{t,s_1,\dots,s_m,t_1,\dots,t_n\}$;
\item $F'$ be the filter generated by $F\cup\{\neg t\}$ and $G'$ be the filter generated by $G\cup\{s_1,\dots,s_m,\neg t_1,\dots,\neg t_n\}$, which are again proper filters by independence;
\item $Z=A\setminus W$.
\end{itemize}
This completes the proof of the lemma.
\end{proof}

Given $U_\alpha$, $V_\alpha$, and $S_\alpha$, for some $\alpha<2^\kappa$, a double application of Lemma \ref{lemma:kunen} gives $U_{\alpha+1}\supseteq U_\alpha$, $V_{\alpha+1}\supseteq V_\alpha$, and $S_{\alpha+1}\subseteq S_\alpha$ satisfying conditions \eqref{kunenquattro}, \eqref{kunencinque}, and \eqref{kunensei}. 

At the limit stages of the construction, we proceed according to \eqref{kunentre}, which is easily seen to preserve \eqref{kunencinque}. Finally, let $U$ and $V$ be ultrafilters on $\B$ extending $\bigcup_{\alpha<2^\kappa} U_\alpha$ and $\bigcup_{\alpha<2^\kappa} V_\alpha$, respectively. Then $\langle U,\ge\rangle\equiv_{\mathrm{T}}\langle V,\ge\rangle$ by Lemma \ref{lemma:max} and condition \eqref{kunenuno}. To see that $U$ and $V$ are JPR-incomparable, suppose for example that $U\le_\mathrm{JPR} V$. In particular, this implies the existence of $f\in A^{[\B]}$ such that for every $X\subseteq A$
\[
\bigvee X\in U\iff\bigvee f^{-1}[X]\in V.
\]
Let $\alpha<2^\kappa$ be such that $f=f_\alpha$; condition \eqref{kunensei} then gives a contradiction.
\end{proof}

\section{M-incomparable ultrafilters}\label{section:quattro}

Our next goal is to construct, assuming the Continuum Hypothesis, two ultrafilters which are JPR-equivalent but M-incomparable. For a complete Boolean algebra $\B$, let us define
\[
\End(\B)=\Set*{h\in{^\B}\B}{\exists b>\bbot\text{ such that }h\colon\B\to\B\rest b\text{ is a complete homomorphism}}.
\]
We have
\[
\abs{\B}\le\abs{\End(\B)}\le 2^{\abs{\B}};
\]
indeed, for the left inequality note that, given $b>\bbot$, the function
\[
\begin{split}
\B &\longrightarrow\B\rest b \\
a &\longmapsto a\wedge b
\end{split}
\]
belongs to $\End(\B)$. The other inequality is clear.

The following lemma provides a better upper bound on the cardinality of $\End(\B)$ for some quotient Boolean algebras. Recall that $\mathcal{B}(^\omega2)$ denotes the Borel $\sigma$-algebra of the Cantor space $^\omega2$.

\begin{lemma}\label{lemma:end} Let $I$ be a c.c.c.\ $\aleph_1$-complete ideal over $^\omega2$, let $\B=\mathcal{B}(^\omega2)/I$ be the corresponding quotient algebra. Then $\B$ is complete and $\abs{\End(\B)}\le 2^{\aleph_0}$.
\end{lemma}
\begin{proof} The completeness of $\B$ is a consequence of a classic result of Smith and Tarski \cite[Corollary 4.3]{smithtarski:distr}. To see that $\abs{\End(\B)}\le 2^{\aleph_0}$, if $h,h'\in\End(\B)$ have the property that for every clopen $O\subseteq{^\omega2}$
\[
h\bigl({[O]}_I\bigr)=h'\bigl({[O]}_I\bigr),
\]
then necessarily $h=h'$. Indeed, a Borel set is obtained from clopen sets applying the operations of complement and countable union, and every complete homomorphism respects those two operations. In conclusion, a function $h\in\End(\B)$ is uniquely determined by its values on the countable algebra of clopen sets, hence there are at most $2^{\aleph_0}$ possibilities.
\end{proof}

\begin{definition} Let $\B$ be a complete Boolean algebra, $F$ be a filter, and $A$ be a maximal antichain. We say that $F$ is \emph{based} on $A$ if for all $u\in F$
\[
\bigvee\Set{a\in A}{a\le u}\in F.
\]
\end{definition}

We now work towards the main result of the section. Two lemmas are required, the first dealing with the successor stage of the recursive construction.

\begin{lemma}\label{lemma:succ} Let $\B$ be a complete atomless Boolean algebra. Suppose we are given two filters $F$ and $G$ on $\B$, a maximal antichain $A$, and a bijection $f\colon A\to A$, satisfying the following conditions:
\begin{itemize}
\item $F$ and $G$ are based on $A$;
\item $f=f^{-1}$ and for every $X\subseteq A$ we have $\bigvee X\in F\iff\bigvee f[X]\in G$.
\end{itemize}
Let $y\in\B$ be such that $\bigvee\Set{a\in A}{a\wedge y>\bbot}\in F$ and let $z\in\B$ be such that $G\cup\{z\}$ has the finite intersection property. Then there exist filters $F'$ and $G'$, a maximal antichain $A'$, and a bijection $f'\colon A'\to A'$ such that:
\begin{itemize}
\item $F\cup\{y\}\subseteq F'$ and $G\cup\{z\}\subseteq G'$;
\item $\cof(F)=\cof(F')$ and $\cof(G)=\cof(G')$;
\item $F'$ and $G'$ are based on $A'$;
\item $f'={f'}^{-1}$ and for every $X\subseteq A'$ we have $\bigvee X\in F'\iff\bigvee f'[X]\in G'$;
\item for every $b\in A'$ there exists $a\in A$ such that $b\le a$ and $f'(b)\le f(a)$.
\end{itemize}
\end{lemma}
\begin{proof} For each $a\in A$, we shall define $a^0,a^1\in\B$ by splitting $a$ according to $y$. More precisely, suppose both
\[
a\wedge y>\bbot\quad\text{and}\quad a\wedge \neg y>\bbot;
\]
then let $a^0=a\wedge y$ and $a^1=a\wedge \neg y$. Suppose not, then by atomlessness let $a^0$ and $a^1$ be any non-zero elements such that $a^0\vee a^1=a$ and $a^0\wedge a^1=\bbot$.

Again, for each $a\in A$ we define $a^{00},a^{01}\in\B$ by further splitting $a^0$ according to $z$: if both
\[
a^0\wedge z>\bbot\quad\text{and}\quad a^0\wedge \neg z>\bbot,
\]
then let $a^{00}=a^0\wedge z$ and $a^{01}=a^0\wedge \neg z$; otherwise let $a^{00}$ and $a^{01}$ be any non-zero elements such that $a^{00}\vee a^{01}=a^0$ and $a^{00}\wedge a^{01}=\bbot$. Analogously, we define $a^{10},a^{11}\in\B$ by further splitting $a^1$ according to $z$.

Finally let
\[
A'=\Set*{a^{k\ell}}{a\in A,\ k<2,\ \ell<2}
\]
and let $f'\colon A'\to A'$ be defined as follows: for each $a\in A$
\[
f'\bigl(a^{00}\bigr)={f(a)}^{00},\quad f'\bigl(a^{01}\bigr)={f(a)}^{10},\quad f'\bigl(a^{10}\bigr)={f(a)}^{01},\quad f'\bigl(a^{11}\bigr)={f(a)}^{11}.
\]
Clearly $f'$ is a bijection equal to its inverse. By our construction of $A'$, note that there exist $Y\subseteq A'$ and $Z\subseteq A'$ such that $y=\bigvee Y$ and $z=\bigvee Z$.

We claim that
\[
F\cup\Bigl\{y,\bigvee f'[Z]\Bigr\}
\]
has the finite intersection property. Suppose not, then $\neg\bigl(y\wedge\bigvee f'[Z]\bigr)\in F$, therefore
\[
\bigvee\Set*{a\in A}{a\wedge y\wedge\bigvee f'[Z]=\bbot\text{ and }a\wedge y>\bbot}\in F
\]
by the fact that $F$ is a filter based on $A$ and our assumption on $y$. Now let $a\in A$ be such that $a\wedge y\wedge\bigvee f'[Z]=\bbot$ and $a\wedge y>\bbot$. Since $a$ is split according to $y$ we must have $a^0\le a\wedge y$, hence $a^0\wedge\bigvee f'[Z]=\bbot$, which implies $f(a)\wedge z=\bbot$ by the definition of $f'$. In conclusion, we have shown that
\[
\neg z\ge\bigvee\Set*{f(a)}{a\wedge y\wedge\bigvee f'[Z]=\bbot\text{ and }a\wedge y>\bbot}\in G,
\]
contradicting the hypothesis that $G\cup\{z\}$ has the finite intersection property.

The same argument demonstrates that
\[
G\cup\Bigl\{z,\bigvee f'[Y]\Bigr\}
\]
has the finite intersection property. Let $F'$ be the filter generated by $F\cup\bigl\{y,\bigvee f'[Z]\bigr\}$ and let $G'$ be the filter generated by $G\cup\bigl\{z,\bigvee f'[Y]\bigr\}$. It only remains to show that if $X\subseteq A'$ then
\begin{equation}\label{eq:uno}
\bigvee X\in F'\iff\bigvee f'[X]\in G'.
\end{equation}
Suppose $\bigvee X\in F'$, then there exists $u\in F$ such that $\bigvee X\ge u\wedge y\wedge\bigvee f'[Z]$. Hence, using the fact that $F$ is based on $A$,
\[
\bigvee f'[X]\ge\bigvee\Set{f(a)}{a\le u}\wedge\bigvee f'[Y]\wedge z\in G
\]
as desired. On the other hand, if $\bigvee f'[X]\in G'$ then we obtain $\bigvee X\in F'$ analogously. This establishes \eqref{eq:uno} and completes the proof.
\end{proof}

\begin{remark} In the statement of Lemma \ref{lemma:succ}, the assumption that $\bigvee\Set{a\in A}{a\wedge y>\bbot}\in F$ cannot be weakened to the assumption that $F\cup\{y\}$ has the finite intersection property. To see why, let $A=\Set{a_i}{i<\omega}$ be an arbitrary maximal antichain in $\B$, let $f\colon A\to A$ be the identity function, and $F=G=\Set{b\in\B}{\Set{i<\omega}{a_i\wedge\neg b>\bbot}\text{ is finite}}$. Taking $y=\bigvee\Set{a_{2i}}{i<\omega}$ and $z=\neg y$, it is easy to see that both $F\cup\{y\}$ and $G\cup\{z\}$ have the finite intersection property, but there do not exist $F'$, $G'$, and $f'\colon A'\to A'$ which satisfy the conditions above.
\end{remark}

The second lemma will take care of the limit stages of the recursive construction.

\begin{lemma}\label{lemma:limit} Let $\B$ be a complete Boolean algebra and $F$ a filter on $\B$. Suppose we are given, for each $n<\omega$, a countable maximal antichain $A_n$ in $\B$ and a bijection $f_n\colon A_n\to A_n$, satisfying the following conditions:
\begin{itemize}
\item $F=\Set{b\in\B}{\Set{a\in A_0}{a\wedge\neg b>\bbot}\text{ is finite}}$ and $f_0\colon A_0\to A_0$ is the identity function;
\item for every $n<\omega$, $f_n=f_n^{-1}$;
\item if $m<n<\omega$, then $\bigvee\Set*{b\in A_n}{\exists a\in A_m\bigl (b\le a\text{ and }f_n(b)\le f_m(a)\bigr)}\in F$.
\end{itemize}
Then there exist a countable maximal antichain $A_\infty$ in $\B$ and a bijection $f_\infty\colon A_\infty\to A_\infty$ such that:
\begin{itemize}
\item $f_\infty=f_\infty^{-1}$;
\item if $n<\omega$, then $\bigvee\Set*{b\in A_\infty}{\exists a\in A_n\bigl(b\le a\text{ and }f_\infty(b)\le f_n(a)\bigr)}\in F$.
\end{itemize}
\end{lemma}
\begin{proof} Let us enumerate $A_0=\Set{a_i}{i<\omega}$. By hypothesis we may find a strictly increasing sequence $\Seq{k_n}{n<\omega}$ such that $k_0=0$ and for every $m\le n$
\[
\bigvee_{i\ge k_n}a_i\le\bigvee\Set*{b\in A_n}{\exists a\in A_m\bigl(b\le a\text{ and }f_n(b)\le f_m(a)\bigr)}.
\]
Now define $A_\infty$ diagonally: for every $n<\omega$ let
\[
D_n=\Set[\Bigg]{b\in A_n}{b\le\bigvee_{k_n\le i<k_{n+1}}a_i}
\]
and then let
\[
A_\infty=\bigcup_{n<\omega}D_n.
\]
Easily $A_\infty$ is a countable antichain; to see it is maximal, let $x>\bbot$ be arbitrary. Then there exists $i<\omega$ such that $a_i\wedge x>\bbot$. Let $n<\omega$ be such that $k_n\le i<k_{n+1}$; then clearly $x$ must meet some $b\in D_n$.

We proceed with the definition of $f_\infty$, starting with the observation that for every $n<\omega$ we have $f_n[D_n]=D_n$. Indeed, to see that $f_n[D_n]\subseteq D_n$, let $b\in D_n$. By definition, there exists $k_n\le i<k_{n+1}$ such that $b\le a_i$. Then we have $f_n(b)\le f_0(a_i)=a_i$, because $f_0$ is the identity function. Furthermore, to check $f_n[D_n]=D_n$, simply use the fact that $f_n=f_n^{-1}$.

Motivated by the above observation, we define
\[
f_\infty=\bigcup_{n<\omega}(f_n\rest D_n)
\]
which is easily seen to be a bijection equal to its inverse. By construction
\[
\bigvee_{i\ge k_n}a_i\le\bigvee\Set*{b\in A_\infty}{\exists a\in A_n\bigl(b\le a\text{ and }f_\infty(b)\le f_n(a)\bigr)},
\]
as desired.
\end{proof}

\begin{theorem}\label{theorem:m} Assume $2^{\aleph_0}=\aleph_1$. Let $\B$ be a complete c.c.c.\ Boolean algebra such that $\abs{\End(\B)}\le\aleph_1$. If $\B$ is not atomic, then there exist two ultrafilters $U$ and $V$ on $\B$ such that $U\equiv_\mathrm{JPR} V$ and $U\nleq_\mathrm{M} V$ and $V\nleq_\mathrm{M} U$.
\end{theorem}
\begin{proof} Without loss of generality, $\B$ is atomless. Indeed, if we let
\[
b=\neg\bigvee\Set{a\in\B}{a\text{ is an atom}},
\]
then $\B\rest b$ is an atomless c.c.c.\ Boolean algebra such that $\abs{\End(\B\rest b)}=\aleph_1$. Consequently, if we build $U$ and $V$ on $\B\rest b$, then by Lemma \ref{lemma:dense}\eqref{lemma:jprdue} they will generate ultrafilters on $\B$ with the desired properties.

Hence, assuming $\B$ is atomless, we enumerate
\[
\B=\Set{b_\alpha}{\alpha<\omega_1}
\]
and
\[
\End(\B)=\Set{h_\alpha}{\alpha<\omega_1}
\]
and
\[
\Part(\B)=\Set{P_\alpha}{\alpha<\omega_1}.
\]
We shall carry out a recursive construction of filters $\Seq{U_\alpha}{\alpha<\omega_1}$ and $\Seq{V_\alpha}{\alpha<\omega_1}$ on $\B$, together with maximal antichains $\Seq{A_\alpha}{\alpha<\omega_1}$ in $\B$ and functions $\Seq{f_\alpha}{\alpha<\omega_1}$, according to the following conditions:
\begin{enumerate}
\item\label{diffuno} $U_0=V_0=\Set{b\in\B}{\Set{a\in A_0}{a\wedge\neg b>\bbot}\text{ is finite}}$ and $f_0\colon A_0\to A_0$ is the identity function;
\item\label{diffdue} for every $\alpha<\omega_1$, both $U_\alpha$ and $V_\alpha$ are based on $A_\alpha$;
\item\label{difftre} for every $\alpha<\omega_1$, $f_\alpha\colon A_\alpha\to A_\alpha$ is a bijection, $f_\alpha=f_\alpha^{-1}$, and for every $X\subseteq A_\alpha$ we have $\bigvee X\in U_\alpha\iff\bigvee f_\alpha[X]\in V_\alpha$;
\item\label{diffquattro} if $\alpha<\beta<\omega_1$ then $U_\alpha\subseteq U_\beta$, $V_\alpha\subseteq V_\beta$, and $\bigvee\Set[\big]{b\in A_\beta}{\exists a\in A_\alpha\bigl(b\le a\text{ and }f_\beta(b)\le f_\alpha(a)\bigr)}\in U_0$;
\item\label{diffcinque} if $\delta<\omega_1$ is a limit ordinal, then $U_\delta=\bigcup_{\alpha<\delta}U_\alpha$ and $V_\delta=\bigcup_{\alpha<\delta}V_\alpha$;
\item\label{diffsei} for every $\alpha<\omega_1$, $A_{\alpha+1}$ is a refinement of $P_\alpha$;
\item\label{diffsette} for every $\alpha<\omega_1$, either $b_\alpha\in U_{\alpha+1}$ or $\neg b_\alpha\in U_{\alpha+1}$, and similarly either $b_\alpha\in V_{\alpha+1}$ or $\neg b_\alpha\in V_{\alpha+1}$;
\item\label{diffotto} for every $\alpha<\omega_1$, there exists $u\in U_{\alpha+1}$ such that $h_\alpha(\neg u)\in V_{\alpha+1}$, and similarly there exists $v\in V_{\alpha+1}$ such that $h_\alpha(\neg v)\in U_{\alpha+1}$.
\end{enumerate}

As for the base step, let $A_0$ be any infinite maximal antichain in $\B$. Then $U_0$, $V_0$, and $f_0$, as defined by condition \eqref{diffuno}, clearly satisfy \eqref{diffdue} and \eqref{difftre}.

We take care of the successor step: for $\alpha<\omega_1$, let $U_\alpha$, $V_\alpha$, and $f_\alpha\colon A_\alpha\to A_\alpha$ be given. Let $b\in\B$ be such that for every $a\in A_\alpha$ both $a\wedge b>\bbot$ and $a\wedge\lnot b>\bbot$. Such an element can be easily found by atomlessness, for example by splitting each $a\in A$ in two disjoint positive parts and taking the supremum of the first halves. Since $h_\alpha$ has the property that $h_\alpha(\neg b)=\neg h_\alpha(b)$, it must be the case that $V_\alpha\cup\{h_\alpha(\neg b)\}$ has the finite intersection property or $V_\alpha\cup\{h_\alpha(b)\}$ has the finite intersection property. In the first case let $u=b$, otherwise let $u=\neg b$. By Lemma \ref{lemma:succ}, there exist filters $U'_\alpha$ and $V'_\alpha$, a maximal antichain $A'_\alpha$, and a bijection $f'_\alpha\colon A'_\alpha\to A'_\alpha$ such that $U_\alpha\cup\{u\}\subseteq U'_\alpha$, $V_\alpha\cup\{h_\alpha(\neg u)\}\subseteq V'_\alpha$, and conditions \eqref{diffdue}, \eqref{difftre}, \eqref{diffquattro} are preserved. Symmetrically, by applying Lemma \ref{lemma:succ} again if necessary, we may also assume there exists $v\in V'_\alpha$ such that $h_\alpha(\neg v)\in U'_\alpha$.

Another double application of Lemma \ref{lemma:succ}, taking $y=\btop$ and $z=\pm b_\alpha$, gives filters $U_{\alpha+1}\supseteq U'_\alpha$ and $V_{\alpha+1}\supseteq V'_\alpha$, a maximal antichain $A''_\alpha$, and a bijection $f''_\alpha\colon A''_\alpha\to A''_\alpha$, such that condition \eqref{diffsette} is additionally satisfied.

To complete the successor step, we take care of condition \eqref{diffsei}. Let $A_{\alpha+1}$ be a common refinement of $A''_\alpha$ and $P_\alpha$. Without loss of generality, for each $a\in A''_\alpha$ the set $\Set{b\in A_{\alpha+1}}{b\le a}$ is infinite. This allows to define $f_{\alpha+1}$ as $f''_\alpha$ extended in the obvious way to $A_{\alpha+1}$.

For a limit ordinal $\delta<\omega_1$, let $U_\delta$ and $V_\delta$ be defined according to \eqref{diffcinque}. Since $\delta$ is countable, there exists a strictly increasing sequence $\Seq{\alpha_n}{n<\omega}$ of ordinals such that $\alpha_0=0$ and $\delta=\bigcup_{n<\omega}\alpha_n$. From Lemma \ref{lemma:limit} applied to $\Seq{A_{\alpha_n}}{n<\omega}$, we obtain a maximal antichain $A_\delta$ and a bijection $f_\delta\colon A_\delta\to A_\delta$ such that $f_\delta=f_\delta^{-1}$ and, if $n<\omega$, then $\bigvee\Set*{b\in A_\delta}{\exists a\in A_{\alpha_n}\bigl(b\le a\text{ and }f_\delta(b)\le f_{\alpha_n}(a)\bigr)}\in U_0$. This easily implies that conditions \eqref{diffdue} and \eqref{diffquattro} are preserved.

It only remains to check that, given $X\subseteq A_\delta$, we have $\bigvee X\in U_\delta\iff\bigvee f_\delta[X]\in V_\delta$. Assuming $\bigvee X\in U_\delta$, there exists $n<\omega$ such that $\bigvee X\in U_{\alpha_n}$. Consequently, if we let
\[
X'=\Set*{x\in X}{\exists a\in A_{\alpha_n}\bigl(x\le a\text{ and }f_\delta(x)\le f_{\alpha_n}(a)\bigr)}
\]
then also $\bigvee X'\in U_{\alpha_n}$. This implies, since $U_{\alpha_n}$ is based on $A_{\alpha_n}$, the existence of $Y\subseteq A_{\alpha_n}$ such that $\bigvee X'\ge\bigvee Y\in U_{\alpha_n}$. Then it is easy to check that
\[
\bigvee f_\delta[X]\ge\bigvee f_\delta[X']\ge\bigvee f_{\alpha_n}[Y]\in V_{\alpha_n},
\]
as desired. Conversely, if $\bigvee f_\delta[X]\in V_\delta$ then, by the same argument, we obtain $\bigvee X\in U_\delta$. This completes the recursive construction.

In the end, let $U=\bigcup_{\alpha<\omega_1} U_\alpha$ and $V=\bigcup_{\alpha<\omega_1} V_\alpha$, which are ultrafilters by condition \eqref{diffsette}. Moreover, condition \eqref{diffotto} implies that $U\nleq_\mathrm{M} V$ and $V\nleq_\mathrm{M} U$. To verify that $U\le_\mathrm{JPR} V$, by \ref{lemma:dense}\eqref{lemma:jpruno} it is sufficient to consider a dense subset of $\Part(\B)$, but density is indeed guaranteed by condition \eqref{diffsei}. So let $X\subseteq A_\alpha$ be such that $\bigvee X\in U$; we show that $\bigvee f_\alpha[X]\in V$. Take $\beta>\alpha$ with $\bigvee X\in U_\beta$ and define $Y=\Set*{y\in A_\beta}{\exists x\in X\bigl(y\le x\text{ and }f_\beta(y)\le f_\alpha(x)\bigr)}$. Evidently $\bigvee Y\in U_\beta$ and therefore $\bigvee f_\alpha[X]\ge\bigvee f_\beta[Y]\in V_\beta$ as desired. Furthermore, property \eqref{eq:jprdue} of Definition \ref{definition:jpr} follows from our condition \eqref{diffquattro}. By symmetry of the construction, the same argument gives $V\le_\mathrm{JPR} U$ and completes the proof.
\end{proof}

The combination of Lemma \ref{lemma:end} and Theorem \ref{theorem:m} gives that, assuming $2^{\aleph_0}=\aleph_1$, there exist ultrafilters which are JPR-equivalent but M-incomparable on many Boolean algebras of interest in set theory, including the Cohen and random forcing.

\begin{question} Does the conclusion of Theorem \ref{theorem:m} still hold without the assumption that $2^{\aleph_0}=\aleph_1$?
\end{question}

\section{Tukey-incomparable ultrafilters}\label{section:cinque}

To complete our analysis of the orderings, we aim to construct two ultrafilters which are JPR-equivalent but Tukey-incomparable. This will be achieved using a notion reminiscent of the usual $P$-point property.

\begin{definition}[{Star\'y \cite[Definition 3.1]{stary:coherent}}] Let $\B$ be a complete c.c.c.\ Boolean algebra. An ultrafilter $U$ on $\B$ is a \emph{coherent $P$-ultrafilter} if and only if: for every $P\in\Part(\B)$ and every $\Set{X_n}{n<\omega}\subseteq\mathcal{P}(P)$ such that $\Set{\bigvee X_n}{n<\omega}\subseteq U$, there exists $Y\subseteq P$ such that $\bigvee Y\in U$ and for all $n<\omega$ the set $Y\setminus X_n$ is finite.
\end{definition}

For further details on the relation between coherent $P$-ultrafilters and Tukey reducibility, we refer the reader to our previous work \cite[Section 4]{bp:comb}. The next proposition, essentially due to Ketonen \cite{ketonen:p}, is the main ingredient to construct such ultrafilters. For a detailed proof in the Boolean-algebraic context, see the proof of \cite[Proposition 3.4]{stary:coherent}. As usual, here $\mathfrak{d}$ denotes the minimum cardinality of a dominating family of functions from $\omega$ to $\omega$. 

\begin{proposition}\label{proposition:stary} Let $\B$ be a complete c.c.c.\ Boolean algebra, let $F$ be a filter on $\B$ such that $\cof(F)<\mathfrak{d}$. For every $P\in\Part(\B)$ and every $\Set{X_n}{n<\omega}\subseteq\mathcal{P}(P)$ such that $\Set{\bigvee X_n}{n<\omega}\subseteq F$, there exists $Y\subseteq P$ such that $F\cup\{\bigvee Y\}$ has the finite intersection property and for all $n<\omega$ the set $Y\setminus X_n$ is finite.
\end{proposition}

From now on, we restrict our attention to the \emph{Cohen algebra} $\C_\omega$, the unique atomless complete Boolean algebra with a countable dense subalgebra. Let $\Set{d_k}{k<\omega}$ be an enumeration of the dense subalgebra, with $d_0=\bbot$. Furthermore, let $\Set{a_i}{i<\omega}$ be an infinite maximal antichain in $\C_\omega$, which we also consider fixed throughout the section.

\begin{definition} Given a strictly increasing function $\varphi\colon\omega\to\omega$, let $C_\varphi\subseteq\C_\omega$ be defined as follows: $c\in C_\varphi$ if and only if there exists a strictly increasing sequence $\Seq{m_j}{j<\omega}$ such that for all $j<\omega$
\[
\Set{a_i\wedge c}{i<\varphi(m_j)}\subseteq\Set{d_k}{k<\varphi(m_j)}.
\]
\end{definition}

We observe that the set $C_\varphi$ is naturally closed under initial segments: if $c\in C_\varphi$ and $\ell<\omega$ then $\bigvee\Set{a_i\wedge c}{i<\ell}\in C_\varphi$.

\begin{lemma}\label{lemma:cof} Let $U$ be a coherent $P$-ultrafilter on $\C_\omega$ such that $U\cap\Set{a_i}{i<\omega}=\emptyset$. For every strictly increasing function $\varphi\colon\omega\to\omega$, the set $C_\varphi\cap U$ is cofinal in $U$.
\end{lemma}
\begin{proof} Let $u\in U$; we have to find some $c\in C_\varphi\cap U$ such that $c\le u$. By our \cite[Lemma 4.5]{bp:comb}, there exists $v\in U$ such that $v\le u$ and
\[
\Set{a_i\wedge v}{i<\omega}\subseteq\Set{d_k}{k<\omega}.
\]
Inductively, choose a strictly increasing sequence $\Seq{n_j}{j<\omega}$ such that $n_0=0$ and for all $j<\omega$
\[
\Set{a_i\wedge v}{i\le\varphi(n_j)}\subseteq\Set{d_k}{k<\varphi(n_{j+1})}.
\]

Suppose
\[
z=\bigvee\Set[\big]{a_i}{\text{there exists }j<\omega\text{ such that }\varphi(n_{2j+1})\le i<\varphi(n_{2j+2})}\notin U,
\]
then we let $c=v\wedge\neg z$ and, for all $j<\omega$, let $m_j=n_{2j}$. Then $c\le u$ and $c\in C_\varphi$, as witnessed by the sequence $\Seq{m_j}{j<\omega}$. On the other hand, in case $z\in U$, we let $c=v\wedge z$ and proceed analogously.
\end{proof}

In the proof of the main Theorem \ref{theorem:main}, we aim to construct two coherent $P$-ultrafilters which are Tukey-incomparable. However, a priori there are $2^{2^{\aleph_0}}$ potential Tukey reductions to enumerate, which means a direct recursive construction has little hope to succeed. To get around this problem we first establish that, whenever $U$ is a coherent $P$-ultrafilter on $\C_\omega$, each potential Tukey reduction on $U$ has a canonical representation. This idea dates back to Dobrinen and Todor\v{c}evi\'c \cite[Section 3]{dt:tukey} and has been thoroughly examined by Dobrinen \cite{dobrinen:continuous}.

We introduce a convenient piece of notation: given $c\in\C_\omega$ and $\ell<\omega$, let
\[
c^\ell=\bigvee\Set{a_i\wedge c}{i<\ell}\vee\bigvee\Set{a_i}{i\ge\ell}.
\]

\begin{definition}\label{definition:cont} Let $C\subseteq\C_\omega$; a function $f\colon C\to\C_\omega$ is \emph{represented by a basic map} if there exists a map $\hat{f}\colon\Set*{c^\ell}{c\in C,\ \ell<\omega}\to\C_\omega$ such that for all $c\in C$
\[
f(c)=\bigwedge\Set*{\hat{f}\bigl(c^\ell\bigr)}{\ell<\omega}.
\]
\end{definition}

\begin{remark}\label{remark:cont} If $\varphi\colon\omega\to\omega$ is strictly increasing, then the set $\Set*{c^\ell}{c\in C_\varphi,\ \ell<\omega}$ is countable. It follows that, for every such $\varphi$, there are at most $2^{\aleph_0}$ many functions $f\colon C_\varphi\to\C_\omega$ represented by a basic map.
\end{remark}

\begin{theorem}\label{theorem:rep} Let $U$ be a coherent $P$-ultrafilter on $\C_\omega$ such that $U\cap\Set{a_i}{i<\omega}=\emptyset$. Then, for every monotonic function $g\colon U\to\C_\omega$ there exist $x\in U$, a strictly increasing function $\varphi\colon\omega\to\omega$, and a monotonic function $g^*\colon C_\varphi\to\C_\omega$ represented by a basic map such that $g^*\rest(U\rest x)=g\rest(C_\varphi\cap U\rest x)$.
\end{theorem}
\begin{proof} We generalize the proof of \cite[Theorem 20]{dt:tukey}, using Lemma \ref{lemma:cof} to control the interaction between the maximal antichain and the dense subalgebra. We begin by constructing a subset $\Set{x_n}{n<\omega}$ of $U$ such that, for each $n<\omega$,
\begin{itemize}
\item $x_{n+1}\le x_n$;
\item $x_n\wedge\bigvee\Set{a_i}{i<n}=\bbot$;
\item for every function $s\colon n\to\{d_0,\dots,d_{n-1}\}$ and for every $k<n$, if there exists $u\in U$ such that $d_k\nleq g(u)$ and $s(i)=a_i\wedge u$ for all $i<n$, then $d_k\nleq g(s(0)\vee\dots\vee s(n-1)\vee x_n)$.
\end{itemize}

Recursively, let $x_0=\btop$ and suppose we have $x_{n-1}$ for some $n>0$. Choose first some $w\in U$ such that $w\le x_{n-1}$ and $w\wedge\bigvee\Set{a_i}{i<n}=\bbot$. Now, for each function $s\colon n\to\{d_0,\dots,d_{n-1}\}$ and $k<n$, if there exists $u\in U$ such that $d_k\nleq g(u)$ and $s(i)=a_i\wedge u$ for all $i<n$, then let $u_{s,k}$ be such a $u$. On the other hand, if such $u$ does not exist, then simply let $u_{s,k}=w$. Finally, we define
\[
x_n=w\wedge\bigwedge\Set[\big]{u_{s,k}}{s\colon n\to\{d_0,\dots,d_{n-1}\}\text{ and }k<n}.
\]
To check the sequence we constructed has the desired properties, let $n<\omega$. Clearly the first two conditions present no problems, so let $s\colon n\to\{d_0,\dots,d_{n-1}\}$ and $k<n$ be such that there exists $u\in U$ with $d_k\nleq g(u)$ and $s(i)=a_i\wedge u$ for all $i<n$. By construction, this means $d_k\nleq g(u_{s,k})$, but $s(0)\vee\dots\vee s(n-1)\vee x_n\le u_{s,k}$ and therefore $d_k\nleq g(s(0)\vee\dots\vee s(n-1)\vee x_n)$ by monotonicity.

Since $U$ is a coherent $P$-ultrafilter, by \cite[Lemma 4.3]{bp:comb} there exists $y\in U$ such that for each $n<\omega$ the set $\Set{i<\omega}{a_i\wedge y\wedge\neg x_n>\bbot}$ is finite. Accordingly, we can find a strictly increasing sequence $\Seq{n_j}{j<\omega}$ such that $n_0=0$ and for all $j<\omega$
\[
\bigvee\Set{a_i\wedge y}{i\ge n_{j+1}}\le x_{n_j}.
\]
Suppose
\[
z=\bigvee\Set[\big]{a_i}{\text{there exists }j<\omega\text{ such that }n_{2j+1}\le i<n_{2j+2}}\notin U,
\]
then we let $x=y\wedge\neg z$ and, for all $j<\omega$, let $\varphi(j)=n_{2j+1}$. In the case of $z\in U$, of course we would let $x=y\wedge z$, $\varphi(j)=n_{2j}$, and proceed similarly. Observe that, by definition, for each $j<\omega$
\begin{equation}\label{eq:due}
\bigvee\Set{a_i\wedge x}{i\ge\varphi(j)}=\bigvee\Set{a_i\wedge x}{i\ge n_{2j+2}}\le \bigvee\Set{a_i\wedge y}{i\ge n_{2j+2}}\le x_{\varphi(j)}.
\end{equation}

Next, we aim to show that for all $c\in C_\varphi\cap U\rest x$
\begin{equation}\label{eq:tre}
g(c)=\bigwedge\Set*{g\bigl(c^\ell\wedge x\bigr)}{\ell<\omega}.
\end{equation}
Clearly, if $\ell<\omega$ then $g(c)\le g\bigl(c^\ell\wedge x\bigr)$ by monotonicity. To prove the reverse inequality, our strategy is to show that for all $k<\omega$ if $d_k\le\bigwedge\Set*{g\bigl(c^\ell\wedge x\bigr)}{\ell<\omega}$ then $d_k\le g(c)$. So let us fix $k$ such that $d_k\le g\bigl(c^\ell\wedge x\bigr)$ for all $\ell<\omega$. Since $c\in C_\varphi$, there exists a strictly increasing sequence $\Seq{m_j}{j<\omega}$ such that for all $j<\omega$
\[
\Set{a_i\wedge c}{i<\varphi(m_j)}\subseteq\bigl\{d_0,\dots,d_{\varphi(m_j)-1}\bigr\}.
\]
Now fix a sufficiently large $j$ such that $k<\varphi(m_j)$. Letting $s\colon\varphi(m_j)\to\bigl\{d_0,\dots,d_{\varphi(m_j)-1}\bigr\}$ be the function defined by $s(i)=a_i\wedge c$, we have the chain of implications
\[
d_k\le g\bigl(c^{\varphi(m_j)}\wedge x\bigr)\implies d_k\le g\bigl(s(0)\vee\dots\vee s(\varphi(m_j)-1)\vee x_{\varphi(m_j)}\bigr)\implies d_k\le g(c),
\]
where the first implication follows from \eqref{eq:due} and monotonicity of $g$, and the second follows from the third condition in the construction of $\Set{x_n}{n<\omega}$. Therefore, $d_k\le g(c)$ and \eqref{eq:tre} is established.

Let $g^*\colon C_\varphi\to\C_\omega$ be defined as follows: for every $c\in C_\varphi$
\[
g^*(c)=\bigwedge\Set*{g\bigl(c^\ell\wedge x\bigr)}{\ell<\omega}.
\]
From this definition, it follows that $g^*$ is represented by the basic map given by $c^\ell\mapsto g\bigl(c^\ell\wedge x\bigr)$. Since $g$ is monotonic, clearly $g^*$ is monotonic as well. Finally, \eqref{eq:tre} implies that $g^*\rest(U\rest x)=g\rest(C_\varphi\cap U\rest x)$.
\end{proof}

We move on to present the main result of this section.

\begin{theorem}\label{theorem:main} Assume $2^{\aleph_0}=\aleph_1$. Then there exist two ultrafilters $U$ and $V$ on $\C_\omega$ such that $U\equiv_\mathrm{JPR} V$ and $\langle U,\ge\rangle\nleq_{\mathrm{T}}\langle V,\ge\rangle$ and $\langle V,\ge\rangle\nleq_{\mathrm{T}}\langle U,\ge\rangle$.
\end{theorem}
\begin{proof} Let us begin by enumerating
\[
\C_\omega=\Set{b_\alpha}{\alpha<\omega_1}.
\]
By Remark \ref{remark:cont}, for each strictly increasing $\varphi\colon\omega\to\omega$ there are no more than $2^{\aleph_0}$ many functions $g^*\colon C_\varphi\to\C_\omega$ which are represented by a basic map: we enumerate all such functions, for all possible $\varphi$, as
\[
\Set{g_\alpha}{\alpha<\omega_1}.
\]
Finally, we enumerate the conditions for being coherent $P$-ultrafilters. More precisely, consider the set of all pairs $\bigl\langle\Set{X_n}{n<\omega},P\bigr\rangle$ such that $P\in\Part(\C_\omega)$ and $X_n\subseteq P$ for all $n<\omega$; we enumerate this set of pairs as $\Set{p_\alpha}{\alpha<\omega_1}$.

We shall carry out a recursive construction of filters $\Seq{U_\alpha}{\alpha<\omega_1}$ and $\Seq{V_\alpha}{\alpha<\omega_1}$ on $\C_\omega$, together with maximal antichains $\Seq{A_\alpha}{\alpha<\omega_1}$ in $\C_\omega$ and functions $\Seq{f_\alpha}{\alpha<\omega_1}$, according to the following conditions:
\begin{enumerate}
\item\label{uno} $A_0=\Set{a_i}{i<\omega}$, $U_0=V_0=\Set{b\in\C_\omega}{\Set{i<\omega}{a_i\wedge\neg b>\bbot}\text{ is finite}}$, and $f_0\colon A_0\to A_0$ is the identity function;
\item\label{due} for every $\alpha<\omega_1$, both $U_\alpha$ and $V_\alpha$ are based on $A_\alpha$ and $\cof(U_\alpha)=\cof(V_\alpha)=\aleph_0$;
\item\label{tre} for every $\alpha<\omega_1$, $f_\alpha\colon A_\alpha\to A_\alpha$ is a bijection, $f_\alpha=f_\alpha^{-1}$, and for every $X\subseteq A_\alpha$ we have $\bigvee X\in U_\alpha\iff\bigvee f_\alpha[X]\in V_\alpha$;
\item\label{quattro} if $\alpha<\beta<\omega_1$ then $U_\alpha\subseteq U_\beta$, $V_\alpha\subseteq V_\beta$, and $\bigvee\Set[\big]{b\in A_\beta}{\exists a\in A_\alpha\bigl(b\le a\text{ and }f_\beta(b)\le f_\alpha(a)\bigr)}\in U_0$;
\item\label{cinque} if $\delta<\omega_1$ is a limit ordinal, then $U_\delta=\bigcup_{\alpha<\delta}U_\alpha$ and $V_\delta=\bigcup_{\alpha<\delta}V_\alpha$;
\item\label{sei} for every $\alpha<\omega_1$, if $p_\alpha=\bigl\langle\Set{X_n}{n<\omega},P\bigr\rangle$, then $A_{\alpha+1}$ is a refinement of $P$;
\item\label{sette} for every $\alpha<\omega_1$, if $p_\alpha=\bigl\langle\Set{X_n}{n<\omega},P\bigr\rangle$, then either there exists $n<\omega$ such that $\neg\bigvee X_n\in U_{\alpha+1}$, or there exists $Y\subseteq P$ such that $\bigvee Y\in U_{\alpha+1}$ and for every $n<\omega$ the set $Y\setminus X_n$ is finite; similarly for $V_{\alpha+1}$;
\item\label{otto} for every $\alpha<\omega_1$, either $b_\alpha\in U_{\alpha+1}$ or $\neg b_\alpha\in U_{\alpha+1}$, and similarly either $b_\alpha\in V_{\alpha+1}$ or $\neg b_\alpha\in V_{\alpha+1}$;
\item\label{nove} for every $\alpha<\omega_1$, either there exists $v\in\dom(g_\alpha)\cap V_{\alpha+1}$ such that $\neg g_\alpha(v)\in U_{\alpha+1}$, or there exists $u\in U_{\alpha+1}$ such that for all $x\in\dom(g_\alpha)$ if $V_{\alpha}\cup\{x\}$ has the finite intersection property then $g_\alpha(x)\nleq u$, and similarly reversing the roles of $U$ and $V$.
\end{enumerate}

As usual, the base step presents no problems. For $\alpha<\omega_1$, let $U_\alpha$, $V_\alpha$, and $f_\alpha\colon A_\alpha\to A_\alpha$ be given. Let $b\in\C_\omega$ be an element such that for every $a\in A_\alpha$ both $a\wedge b>\bbot$ and $a\wedge\lnot b>\bbot$. Suppose there exists $x\in\dom(g_\alpha)$ such that $V_\alpha\cup\{x\}$ has the finite intersection property and $g_\alpha(x)\wedge b=\bbot$; then let $u=b$ and $v=x$. Suppose not, which means for all $x\in\dom(g_\alpha)$ if $V_{\alpha}\cup\{x\}$ has the finite intersection property then $g_\alpha(x)\wedge b>\bbot$; then let $u=\neg b$ and $v=\btop$. By Lemma \ref{lemma:succ}, there exist filters $U'_\alpha$ and $V'_\alpha$, a maximal antichain $A'_\alpha$, and a bijection $f'_\alpha\colon A'_\alpha\to A'_\alpha$ such that $U_\alpha\cup\{u\}\subseteq U'_\alpha$, $V_\alpha\cup\{v\}\subseteq V'_\alpha$, and conditions \eqref{due}, \eqref{tre}, \eqref{quattro} are preserved. Symmetrically, by applying Lemma \ref{lemma:succ} once more, we can ensure the same in the opposite direction, thus fulfilling condition \eqref{nove}.

Next, to deal with condition \eqref{sette}, suppose $p_\alpha=\bigl\langle\Set{X_n}{n<\omega},P\bigr\rangle$. In case there exists $n<\omega$ such that $\bigvee X_n\notin U'_\alpha$, we let $y=\neg\bigvee X_n$. In case there is no such $n$, since $\cof(U'_\alpha)=\aleph_0<\mathfrak{d}$ we may apply Proposition \ref{proposition:stary} to obtain $Y\subseteq P$ such that $U'_\alpha\cup\{\bigvee Y\}$ has the finite intersection property and for all $n<\omega$ the set $Y\setminus X_n$ is finite; in this case we let $y=\bigvee Y$. By Lemma \ref{lemma:succ}, there exist filters $U''_\alpha$ and $V''_\alpha$, a maximal antichain $A''_\alpha$, and a bijection $f''_\alpha\colon A''_\alpha\to A''_\alpha$ such that $U'_\alpha\cup\{y\}\subseteq U''_\alpha$, $V'_\alpha\subseteq V''_\alpha$, and conditions \eqref{due}, \eqref{tre}, \eqref{quattro} are preserved. We then repeat this construction for $V''_\alpha$.

To complete the successor step, it remains to guarantee conditions \eqref{otto} and \eqref{sei}, but this can be done exactly as for Theorem \ref{theorem:m}. Furthermore, at the limit step we apply Lemma \ref{lemma:limit} and argue as before.

In conclusion, let $U=\bigcup_{\alpha<\omega_1} U_\alpha$ and $V=\bigcup_{\alpha<\omega_1} V_\alpha$, which are JPR-equivalent ultrafilters. To see that they are also Tukey-incomparable, suppose for example that $\langle U,\ge\rangle\le_{\mathrm{T}}\langle V,\ge\rangle$. In particular, by Remark \ref{remark:monotonic} there exists a monotonic function $g\colon V\to U$ such that $g[V]$ is cofinal in $U$. By conditions \eqref{uno} and \eqref{sette}, $V$ is a coherent $P$-ultrafilter such that $V\cap\Set{a_i}{i<\omega}=\emptyset$, therefore Theorem \ref{theorem:rep} provides $y\in V$, a strictly increasing $\varphi\colon\omega\to\omega$, and a monotonic function $g^*\colon C_\varphi\to\C_\omega$ represented by a basic map such that $g^*\rest(V\rest y)=g\rest(C_\varphi\cap V\rest y)$. Keeping in mind that $g^*=g_\alpha$ for some $\alpha<\omega_1$, condition \eqref{nove} gives two possibilities.

The first is that there exists $v\in C_\varphi\cap V$ such that $\neg g^*(v)\in U$. If this happens, by Lemma \ref{lemma:cof} there exists $c\in C_\varphi\cap V$ such that $c\le v\wedge y$. Therefore $g(c)=g^*(c)\le g^*(v)$ and so $g^*(v)\in U$, a contradiction. The second possibility is that there exists $u\in U$ such that for all $x\in C_\varphi\cap V$ we have $g^*(x)\nleq u$. Then, using the fact that $g[V]$ is cofinal in $U$, let $v\in V$ be such that $g(v)\le u$. Again, by taking some $c\in C_\varphi\cap V$ such that $c\le v\wedge y$, we obtain $g^*(c)=g(c)\le g(v)\le u$, a contradiction. This shows that $U$ and $V$ are Tukey-incomparable and completes the proof.
\end{proof}

Combined with Proposition \ref{proposition:27}, the above theorem gives in particular two ultrafilters which are JPR-equivalent but M-incomparable. In this sense, its conclusion is stronger than Theorem \ref{theorem:m}, which on the other hand had more flexibility in the choice of the Boolean algebra.

\begin{question} Is the conclusion of Theorem \ref{theorem:main} still true under the weaker assumption that $\mathfrak{d}=2^{\aleph_0}$?
\end{question}

\end{document}